\documentclass[11pt,a4paper]{article}
\usepackage{fullpage}
\usepackage{theorem,graphicx,amssymb,amsmath,url}

\theorembodyfont{\slshape}

\newtheorem{thm}{Theorem}
\newtheorem{lem}[thm]{Lemma}

\newtheorem{obs}{Observation}

\def\QED{\ensuremath{{\square}}}
\def\markatright#1{\leavevmode\unskip\nobreak\quad\hspace*{\fill}{#1}}
\newenvironment{proof}
 {\begin{trivlist}\item[\hskip\labelsep{\bf Proof.}]}
 {\markatright{\QED}\end{trivlist}}

\usepackage{xcolor}
\usepackage[inline]{enumitem}
\usepackage{hyperref}

\DeclareMathOperator{\TL}{TL}

\begin{document}

\title{Perfect $k$-colored matchings and $(k\!+\!2)$-gonal tilings}

\author{
	Oswin Aichholzer\thanks{Institute of Software Technology,
	Graz University of Technology, Graz, Austria, \newline  {\tt [oaich|bvogt]@ist.tugraz.at} } 
\and
    Lukas Andritsch\thanks{Mathematics and Scientific Computing, University of Graz, Graz, Austria, \newline  {\tt [baurk|lukas.andritsch]@uni-graz.at} }
\and
	Karin Baur\textsuperscript{$\dagger$}
\and
	Birgit Vogtenhuber\textsuperscript{$\ast$}
}

\maketitle

\begin{abstract}
  We derive a simple bijection between geometric plane perfect
  matchings on $2n$ points in convex position and triangulations on
  $n+2$ points in convex position. We then extend this bijection to
  monochromatic plane perfect matchings on periodically $k$-colored
  vertices and $(k+2)$-gonal tilings of convex point sets. These
  structures are related to 
  a generalization of Temperley-Lieb algebras and our bijections
  provide explicit one-to-one relations between matchings and tilings.
  Moreover, for a given element of one class, the corresponding element of the
  other class can be computed in linear time.
\end{abstract}

\section{Introduction}

The Fuss-Catalan numbers $f(k,m)=\tfrac{1}{m}  {km+m \choose m-1}$ are known to count the number of $(k\!+\!2)$-gonal tilings of a convex polygon of size $km+2$ and go back to Fuss-Euler (cf.~\cite{PS2000}). 
Bisch and Jones introduced $k$-colored Fuss-Catalan algebras in~\cite{BJ1997} as a natural generalization of Temperley-Lieb algebras. 
These algebras have bases by certain planar $k$-colored diagrams with $mk$ vertices on top and bottom. 
The dimension of such an algebra is $f(k,m)$, with a basis indexed by these diagrams. 
We call these diagrams plane perfect $k$-colored matchings or just $k$-colored matchings, 
assuming from now on that they are plane and perfect. 
Since the number of $(k\!+\!2)$-gonal tilings coincides with the number of $k$-colored matchings, these sets are in bijection. 
Przytycki and Sikora~\cite{PS2000} prove this through an inductive implicit construction but do not give an explicit bijection between the sets. 

Furthermore, from work of Marsh and Martin~\cite{MM2006}, one can derive an implicit correspondence between triangulations and diagrams for $k\!=\!1$. 
However, to our knowledge, no explicit bijection is known. 

In this paper, we will give bijections between these two sets of plane graphs on sets of points in convex position. 
We first address the case $k=1$ (Section~\ref{sec:uncolored}) and then treat the general case. 
Our main theorems are the explicit bijections between the set of $k$-colored matchings and $(k+2)$-gonal tilings (Theorems~\ref{thm:bij-1} and \ref{thm:bij-k}). 
A key ingredient is the characterization of valid $k$-colored matchings in Theorem~\ref{thm:valid}. 

\section{Algebraic Background}

\subsection{Temperley-Lieb algebras} 

Temperley and Lieb introduced in \cite{TL1971} an algebra arising from a special kind of lattice models, which is a key ingredient in statistical mechanics.   
Given a field $K$ and an element $\alpha \in K$, the Temperley-Lieb algebra $TL_n(\alpha)$ is the algebra with identity $I$ with generators $u_1,\ldots, u_{n-1}$ , $I$, subject to the relations:
\begin{align}
u_i^2&=\alpha u_i, \ \ 1 \leq i \leq n-1\\
u_i u_j &= u_j u_i, \ \ |i-j| > 1, \  1\leq i,j\leq n-1 \\
u_i u_{i+1} u_i &= u_i,\ \ 1 \leq i \leq n-2 \\
u_{i+1}u_i u_{i+1}&=u_{i+1}, \ \ 1 \leq i \leq n-2.
\end{align} 
The basis of the algebra consists of all reduced words, i.e. words which can not be further simplified using the relations. For example, a basis of $\TL_3(\alpha)$ over the field $k$ is $\{I, u_1, u_2, u_1 u_2, u_2 u_1 \}$, independently of the element $\alpha$.
Kauffman introduced a pictorial representation of the Temperley-Lieb algebras in \cite{K1987}. Each generator corresponds to a plane perfect matching with $n$ vertices on the top and bottom of a rectangle labelled $v_1, \ldots, v_n$ and $v_{n+1}, \ldots, v_{2n}$ in clockwise order.
The identity consists of $n$ propagating lines, and the generator $u_i$ consists of $n-2$ propagating lines and two arcs between the pairs $(v_i,v_{i+1})$ and $(v_{2n-i},v_{2n-i+1})$ respectively, see \figurename~\ref{fig:algebraic_background_generators}.
\begin{figure}[htb]
	\centering\includegraphics[scale=1, page=1]{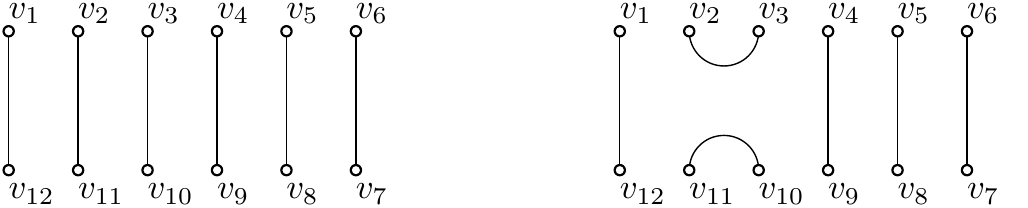}\\
	\caption{The identity $I$ (left) and one of the generators, $u_2$ (right), of $\TL_6(\alpha)$.}
	\label{fig:algebraic_background_generators}
\end{figure}

\noindent Products of generators of the algebra are obtained by concatenation of the corresponding matchings from top to bottom. Any loop arising from this is removed and replaced by a factor $\alpha$, e.g. $u_i u_i= \alpha u_i$, see \figurename~\ref{fig:algebraic_background_loops}.
\begin{figure}[htb]
	\centering\includegraphics[scale=1, page=3]{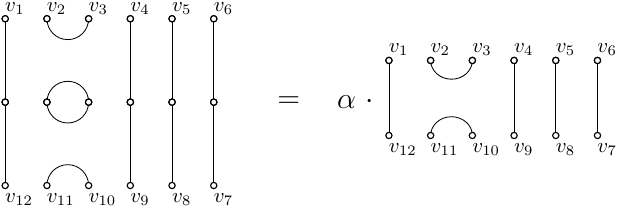}\\
	\caption{Loops are replaced by multiplication with the field element $\alpha$, here: $u_2^2=\alpha u_2$.}
	\label{fig:algebraic_background_loops}
\end{figure}

\noindent   One can check that all the relations $(1)$-$(4)$ are satisfied. Relation $(3)$  is illustrated in \figurename~\ref{fig:algebraic_background_relation}.
\begin{figure}[htb]
	\centering\includegraphics[scale=1, page=2]{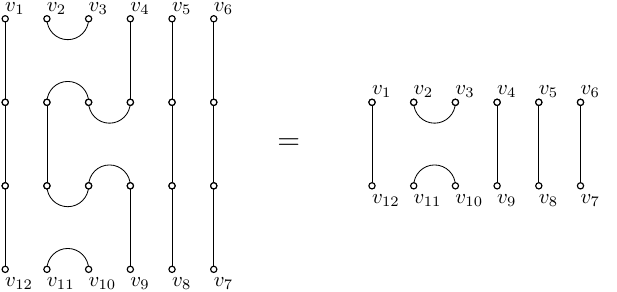}\\
	\caption{The multiplication of generators ($u_2 u_3 u_2$) is shown on the left. The leftmost element of a multiplication always corresponds to the pictogram on the top.}
	\label{fig:algebraic_background_relation}
\end{figure}

\noindent It is a well known result that the dimension of $\TL_n(\alpha)$ is equal to $C_n=\tfrac{1}{n+1}\binom{2n}{n}$, the $n$-th Catalan number (see \cite{BJ1997} for an example). 
We are only interested in the diagrams and will from now on fix $\alpha=1$.

\subsection{Fuss-Catalan algebras}

In \cite{BJ1997}, Bisch and Jones introduced a natural generalization of the Temperley-Lieb algebras, the so called $k$-colored Fuss-Catalan algebras.
These algebras, which we denote by $\TL_{mk,k}(\alpha_1, \ldots, \alpha_k)$, can be defined using the same pictorial representation, now with $mk$ vertices on the top and bottom. However, the diagrams giving the basis must satisfy a further constraint. 
The vertices are colored clockwise starting at the top left vertex, with $k$ colors $c_1, \ldots, c_k$ as follows: $c_1, \ldots, c_{k-1}, c_k, c_k, c_{k-1}, \ldots, c_{2}, c_1, c_{1}, c_2, \ldots, c_k$ and so on. 
Note that the vertices $v_1$ and $v_{2n}$ are always colored with $c_1$ and that the vertices $v_n$ and $v_{n+1}$ have the same color $c_1$ or $c_k$, depending on the parity of $m$. 
In the diagrams, only monochromatic matchings, i.e. matchings where only vertices of the same color are linked, are allowed. 
The identity is again given by straight lines. The generators consist of straight lines and nested sets of arcs as follows: $u_i^{(l)}, 1\leq i \leq m, 1\leq l \leq k$, consists of $l$ nested arcs, where the innermost arc connects vertices $v_{ki}$ and $v_{ki+1}$ and has color $c_{1}$ for $i$ even and $c_{k}$ for $i$ odd, respectively, all other lines are straight. 
See \figurename~\ref{fig:algebraic_background_color} for and illustration of some generators in the $3$-colored case.
\begin{figure}[htb]
	\centering
	\includegraphics[scale=1, page=4]{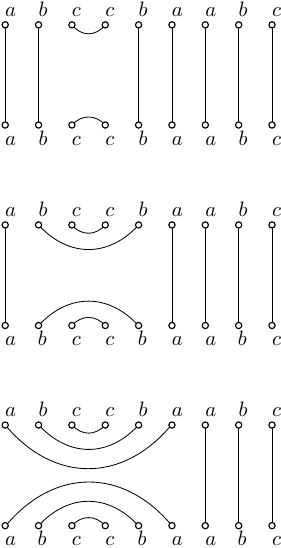}
	\hspace{12ex}
	\includegraphics[scale=1, page=5]{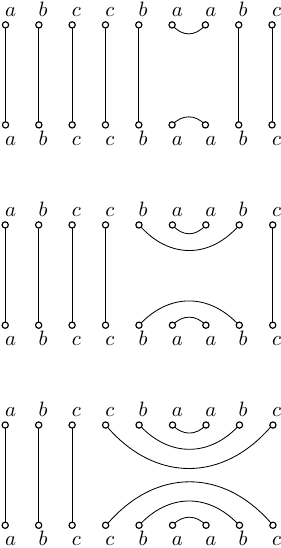}\\
	\caption{The generators of $\TL_{6,3}(\alpha,\beta,\gamma)$ ($m=2$, $k=3$, and $c_1=a,c_2=b,c_3=c$). In the left column, starting in the first row, are the elements $u_1^{(1)}$, $u_1^{(2)}$ and $u_1^{(3)}$, in the right column the elements $u_2^{(1)}$, $u_2^{(2)}$ and $u_2^{(3)}$ respectively. }
	\label{fig:algebraic_background_color}
\end{figure}

Similar to the uncolored case, loops of color $c_i$ correspond to multiplication by a non-zero field element $\alpha_i$. For defining the relations, we follow \cite{D1998}. 
Set $\beta_i(0)=1$ for all $1\leq i \leq m$. Further, for $1 \leq p \leq k$ and $1\leq i \leq m$, set 
\begin{align*}
	\beta_i(p)=\left\{ \begin{array}{ll} \alpha_1\cdot \alpha_2 \cdots \alpha_p & \text{ if} \ i \text{ is even} \\
	\alpha_k\cdot \alpha_{k-1} \cdots \alpha_{k+1-p} & \text{ if} \ i \text{ is odd.} \end{array} \right.
\end{align*}   

\noindent Then the $k$-colored Fuss-Catalan algebra $\TL_{mk,k}(\alpha_1,\ldots,\alpha_k)$ has as generators the identity $I$ and $u_i^{(l)}, 1\leq i \leq m, 1\leq l \leq k$ subject to the relations
\begin{align}
u_i^{(p)}u_i^{(q)}&= u_i^{(q)}u_i^{(p)}= \beta_i(p)u_i^{(q)} \ \text{ if } p \leq q \\
 u_i^{(p)}u_j^{(q)} &= u_j^{(q)}u_i^{(p)} \ \text{ if } |i-j|>1 \ \text{ or } j=i \pm 1 \text{ and } p+q \leq k \\
 u_i^{(p)}u_{i \pm 1}^{(q)}&= \beta_i(k-q)u_i^{(p)}u_{i \pm 1}^{(k-p)} \ \text{ for } p+q>k. 
\end{align}
Note that $u_i^{(0)}=I$ for $1\leq i \leq m$ in these relations whenever needed. All the diagrams generated through this form a basis of $\TL_{mk,k}(\alpha_1,\ldots,\alpha_k)$. The number of basis elements of $\TL_{mk,k}(\alpha_1,\ldots,\alpha_k)$ is  $f(k,m):=\tfrac{1}{m}  {km+m \choose m-1}$ as shown in \cite{BJ1997}. The numbers $f(k,m)$ are  called Fuss-Catalan numbers, a generalization of the Catalan numbers $f(1,m)$. As mentioned above, we are interested in the diagrams and will from now on assume that $\alpha_i=1$ for $1 \leq i \leq k$.

\section{Matchings and triangulations}
\label{sec:uncolored}

In the following, we consider two classes of labeled plane geometric graphs on sets of points in convex position. 
We will tacitly assume that the points are always in convex position and that the graphs are plane. 
The first class are perfect matchings on $2n$ points in convex position. 
We will draw these matchings with two parallel rows of $n$ vertices each, labeled
$v_1$ to $v_{n}$ and $v_{n+1}$ to $v_{2n}$ in clockwise order, and with non-straight edges; see \figurename~\ref{fig:basic_drawing}(left). 
The second class are triangulations on $n+2$ points in convex position, labeled $p_1$ to $p_{n+2}$ in clockwise order; see 
\figurename~\ref{fig:basic_drawing}(right). 
For the sake of distinguishability, throughout this paper we will refer to $p_1, \ldots, p_{n+2}$ as \emph{points} and to $v_1, \ldots, v_{2n}$ as \emph{vertices}.
\begin{figure}[htb]
\centering\includegraphics[scale=1, page=1]{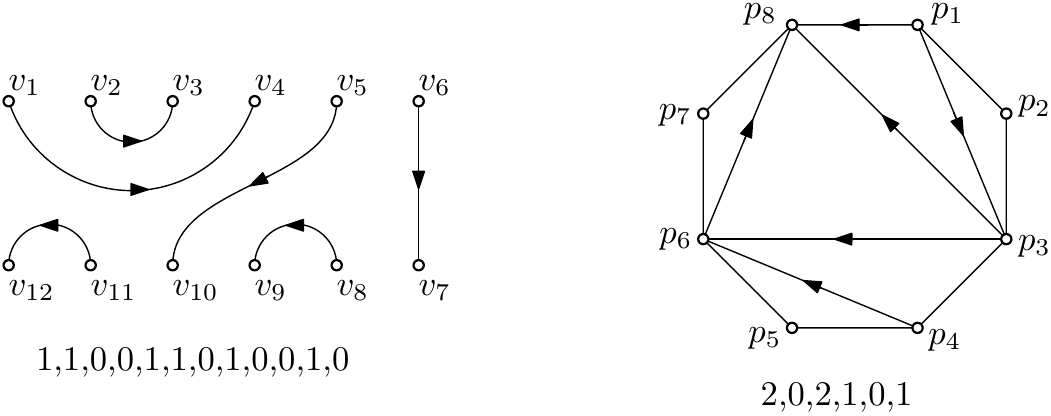}\\
\caption{A perfect matching (left) and the corresponding triangulation for $n=6$ (right).}
\label{fig:basic_drawing}
\end{figure}

The above defined structures are undirected graphs. 
We next give an implicit direction to the edges of these graphs: 
an edge $v_iv_j$ ($p_ip_j$) is directed from $v_i$ to $v_j$ ($p_i$ to $p_j$) for $i<j$, that is, each edge is directed from the vertex / point with lower index to the vertex / point with higher index.
This also defines the outdegree of every vertex / point, which we denote as $b_i$ for each vertex $v_i$ and as $d_i$ for each point $p_i$. 
For technical reasons, 
we do not count the edges on the boundary of the convex hull of a triangulation when computing the outdegree of a point $p_i$, with the exception of the edge $p_1p_{n+2}$.
We call the sequence $(b_1,\ldots, b_{2n})$ of the outdegrees of a matching (or the sequence $(d_1,\ldots, d_n)$ of the first $n$ outdegrees of a triangulation) its \emph{outdegree sequence}; see again \figurename~\ref{fig:basic_drawing}.  
We first show that for both structures, this sequence is sufficient to encode the graph.

For matchings, the outdegree sequence is a $0/1$-sequence with $2n$ digits, where $n$ digits are $1$ and $n$ digits are $0$. 
Moreover, the directions of the edges imply that an incoming edge at a vertex $v_j$ must be outgoing for a vertex $v_i$ with $i<j$. 
Thus, we have the condition $\sum_{i=1}^{\ell} b_i \geq l/2$ for any $1 \leq {\ell} \leq 2n$, that is, in any subsequence 
starting at $v_1$, we have at least as many $1$s as $0$s. 
Such sequences are called ballot sequences; see~\cite[p.69]{F1968}. 
Obviously, the outdegree sequence of a matching can be computed from a given matching in $O(n)$ time. 
But also the reverse is true: 
We consider the outdegrees from $b_1$ to $b_{2n}$.
We use a stack (with the usual push and pop operations) to store the indices of considered vertices that still need to be processed. 
Initially, the stack is empty.
If $b_i=1$, we push the index $i$ on the stack. 
If $b_i=0$, we pop the topmost index ${\ell}$ from the stack and output the edge $v_{\ell}v_i$. 
In this way, always the last vertex with `open' outgoing edge is connected to the next vertex with incoming edge, implying that the subgraph 
with vertices $v_{\ell}$ to $v_i$ is a valid plane perfect matching. 
A simple induction argument shows that the whole resulting graph is plane and can be reconstructed from the outdegree sequence in $O(n)$ time. 

For triangulations, first note that the outdegrees of $p_{n+1}$ and $p_{n+2}$ are 0.
Thus we do not lose information when restricting the outdegree sequence of a triangulation to $(d_1,\ldots, d_n)$. 
As in the previous case, the directions of edges imply that for any valid outdegree sequence, it holds that
$\sum_{i=1}^{\ell} d_{n+1-i} \leq \sum_{i=1}^{\ell} 1 = {\ell}$ for any $1 \leq {\ell} \leq n$.
This sum is precisely the maximum number of edges which can be outgoing from the `last' ${\ell}$ points $p_{n+1-{\ell}}$ to $p_n$. 
Recall that we do not consider the edges of the convex hull, except for $p_1p_{n+2}$, and thus the number of edges which contribute to the outdegree sequence is exactly $n$. 
As before, it is straightforward to compute the outdegree sequence from a given triangulation in $O(n)$ time.
For the reverse process, we again use a stack 
to store the indices of considered points that still need to be processed. 
We initialize the stack with push($n+2$) and push($n+1$) and output all the (non-counted) edges $p_i p_{i+1}$ for $1 \leq i \leq n+1$.
Then we consider the outdegrees in reversed order, that is, 
from $d_n$ to $d_1$. For each degree $d_i$ we perform two steps.
	(1) $d_i$ times, we pop the topmost index from the stack. After each pop let ${\ell}$ be the (new) topmost index on the stack and output the edge $p_ip_{\ell}$. Note that this edge together with the vertex whose index was just popped from the stack forms a triangle of the triangulation we construct.
	(2) We push $i$ on the stack.
This process constructs the triangulation from back to front, i.e., it inserts edges with higher start index first. When processing $p_i$, all points in the range $p_{i+1}$ to $p_{n+2}$ that are still `visible' from $p_i$ (i.e., all points that could still have an incoming edge from $p_i$) 
are in this order on the stack. 
Thus, drawing the edges in the described way generates a planar triangulation. 
At the end of the process, the stack contains exactly the two indices $n+2$ and~$1$, which can be ignored because they are the endpoints of the last generated edge.

So far we have shown that there exist an explicit bijection between outdegree sequences on the one side and matchings respectively triangulations on the other side. 
We now present a bijective transformation between outdegree sequences of matchings and those of triangulations. 

For a given outdegree sequence $B=(b_1, \ldots, b_{2n})$ of a perfect matching,
we compute the outdegree $d_i$ for the point $p_i$ of the  triangulation as the number of 1s between the $(i-1)$-st $0$ and the $i$-th $0$ in~$B$ 
for $i>1$, and set $d_1$ to the number of 1s before the first 0 in~$B$.

For the reverse transformation, we process the outdegree sequence 
of a triangulation from $d_1$ to $d_n$ and set the entries of $B$ in order from $b_1$ to $b_n$ in the following way:
For each entry $d_i$ we first set the next $d_i$ consecutive elements (possibly none) of $B$ 
to 1; then we set the next element of $B$ to 0. 
These 1 elements of $B$ can be regarded as corresponding to the outgoing edges incident with $p_i$, and
	the 0 element regarded as corresponding to the boundary edge adjacent to $p_i$ and going to $p_{i+1}$. 

By the constructions described in the previous two paragraphs it follows immediately that the two transformations are inverse to each other.
Recall that the conditions for valid outdegree sequences are $\sum_{i=1}^{\ell} b_i \geq l/2$ for any $1 \leq {\ell} \leq 2n$ for matchings, and 
$\sum_{i=1}^{\ell} d_{n+1-i} \leq {\ell}$ for any $1 \leq {\ell} \leq n$ for triangulations, respectively.
Having this in mind, it is not hard to see that the two transformations form a bijection between valid outdegree sequences of triangulations and valid outdegree sequences of matchings.
Moreover, each transformation can be performed in $O(n)$ time. 
\figurename~\ref{fig:n5} shows all corresponding perfect matchings, triangulations, and outdegree sequences for $n=3$.
  
\begin{figure}[htb]
\centering\includegraphics[scale=1, page=1]{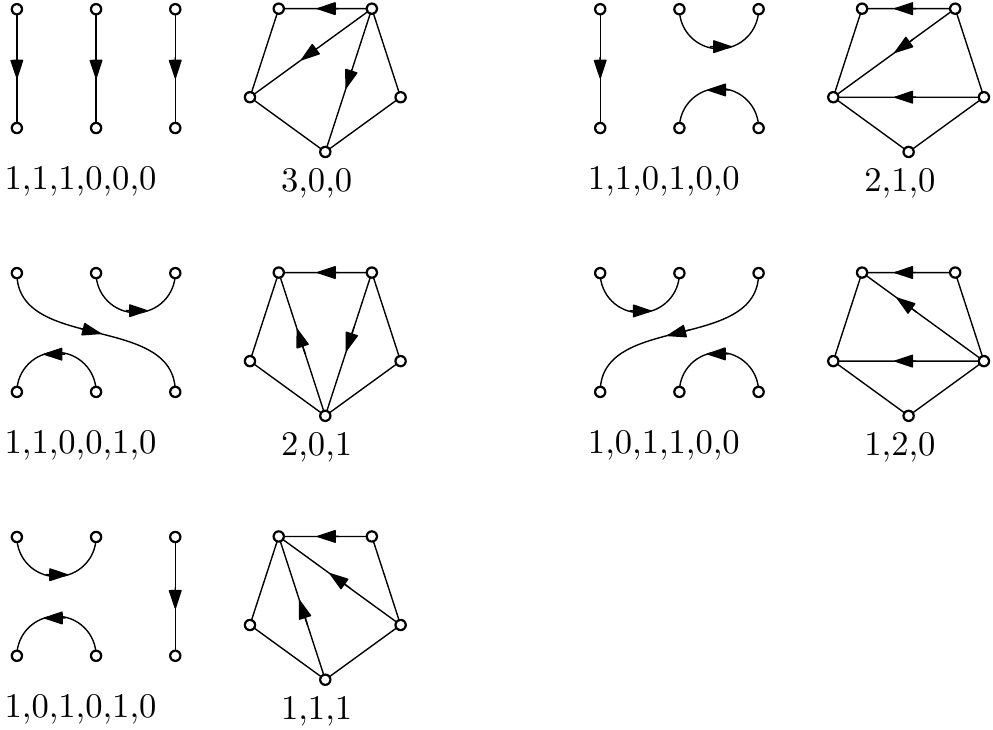}
\caption{All perfect matchings, triangulations, and outdegree sequences for $n=3$.} 
\label{fig:n5}
\end{figure}

\begin{thm}\label{thm:bij-1}
There exists a bijection between geometric plane perfect matchings on $2n$ points in convex position and geometric triangulations on $n+2$ points in convex position. 
Further, for an element of one structure, the corresponding element of the other structure can be computed in linear time.
\end{thm}

\section{Matchings with $k$ colors}
\label{sec:coloredmatchings}

In this section we add colors to the vertices of the perfect matchings and require the matching edges to be monochromatic. 
For $k \geq 2$, let $c_1, \ldots, c_k$ be the $k$ colors and let $n$ be a multiple of $k$. 
We color the vertices in a bitonic way, that is, in the order $c_1, c_2, \ldots, c_{k-1}, c_k, c_k, c_{k-1}, \ldots, c_2, c_1, c_1, c_2, \ldots$ and so on.
In a \emph{perfect $k$-colored matching}, all matching edges connect vertices of the same color, and hence $n$ is a multiple of $k$; see \figurename~\ref{fig:matching_colored} for an example of a $k$-colored matching with $k=3$ colors and $n=9$.
\begin{figure}[htb]
\centering\includegraphics[scale=1,page=1]{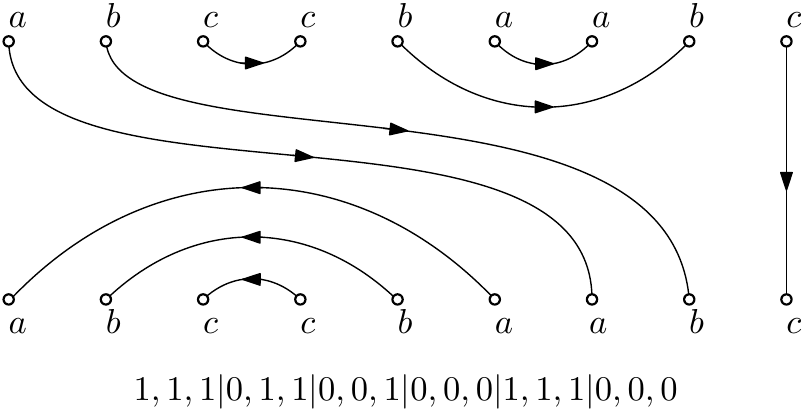}
\caption{Perfect $k$-colored matching for $k=3$ colors and $n=9$ and its outdegree sequence.}
\label{fig:matching_colored}
\end{figure}

Clearly, the set of $k$-colored matching is a subset of the set of non colored matchings considered in the last section, and thus all properties considered there still hold.
But not every matching obtained in the previous section is a $k$-colored matching and hence not every outdegree sequence of a matching is an outdegree sequence of a valid $k$-colored matching.
Thus we now derive additional properties to determine which outdegree sequences of matchings correspond to $k$-colored matchings.

We denote $k$ consecutive vertices $v_i,\ldots,v_{i+k-1}$ that are colored with either $c_1,\ldots,c_k$ or $c_k,\ldots,c_1$ as a \emph{block}. 
In total we have $2n/k$ such blocks and they form a partition of $2n$ vertices. 
Observe that within a block, there cannot be a vertex with an incoming edge after a vertex with an outgoing edge, as this would cause a bichromatic edge. 
Hence, in a $k$-colored matching, the outdegree sequence of any block has to be of the form $|0,\ldots,0,1,\ldots,1|$ (where it can consist entirely of 0 or 1 entries).  
For better readability, we sometimes mark block boundaries in an outdegree sequence with vertical lines.
We say that an outdegree sequence (and the matching) fulfilling this property has a \emph{valid block structure}.

\begin{lem}
\label{lem:shortestedge}
Let $M$ be a perfect matching with valid block structure that is not a $k$-colored matching. Then there exists an edge $v_s v_e$ in $M$ with the following properties: 
\begin{enumerate}[label=(\roman*),partopsep=0ex,topsep=1ex,parsep=0ex,itemsep=1ex]
	\item\label{prop:different_blocks} 
		The vertices $v_s$ and $v_e$ lie in different blocks, say $v_s \in S$ and $v_e \in E$. 
    \item\label{prop:minimal} 
		The subsequence from $v_{s+1}$ to $v_{e-1}$ contains no bichromatic matching edge. 
	\item\label{prop:odd_between} 
		The number of blocks between $S$ and~$E$ is odd.
	\item\label{prop:i_i_plus_1}
		Let $v_s$ be the $i$-th vertex in $S$. Then $v_e$ is the $(i+1)$-st vertex in $E$.
\end{enumerate}
\end{lem}

\begin{proof}
To prove the lemma we assume that $v_s v_e$ is a shortest (with respect to the difference of the indices) edge which connects two vertices of different color and show that any such edge has to fulfill the four properties.\\
{\bf \ref{prop:different_blocks} } As the matching has a valid block structure, no bichromatic edge within a block can exist.\\ 
{\bf \ref{prop:minimal} } If the subsequence from $v_{s+1}$ to $v_{e-1}$ contains a bichromatic matching edge, then this edge is shorter, a contradiction. \\ 
{\bf \ref{prop:odd_between} }	Assume there is an even number of blocks between $S$ and $E$. Then each color shows up in these blocks an even number of times.
	Hence, by Property~\ref{prop:minimal}, the set of vertices in $S$ after $v_s$ has the same set of colors as the set of vertices in $E$ before $v_e$. 
	As $S$ and $E$ are colored in reversed order, this implies that $v_s$ and $v_e$ have the same color, a contradiction. \\ 
{\bf \ref{prop:i_i_plus_1} } As there is an odd number of blocks between $S$ and $E$, by Property~\ref{prop:minimal}, 
	the union of the set of vertices in $S$ after $v_s$ and the set of vertices in $E$ before $v_e$ contains exactly one vertex of each color.  
	As further $S$ and $E$ are colored in the same order, we conclude that the position of $v_e$ in $E$ is 'right after' the position of $v_s$ in~$S$.
\end{proof}

The proof of Lemma~\ref{lem:shortestedge} implies the following theorem.

\begin{thm}
\label{thm:valid}
A matching is a $k$-colored matching if and only if it has a valid block structure and does not contain an edge as described in Lemma~\ref{lem:shortestedge}.
\end{thm}

Remark: For a given outdegree sequence we can check in linear time if
it is an outdegree sequence of a $k$-colored matching by using the
reconstruction algorithm described in Section~\ref{sec:uncolored}.

\section{Tilings with $t$-gons}

For any $t \geq 3$, a \emph{$t$-gonal tiling} or \emph{$t$-angulation} $T$ on $n+2$ points in convex position, labeled $p_1$ to $p_{n+2}$ in clockwise order, is a plane graph where every bounded face is a $t$-gon and the vertices along the unbounded face are $p_1, p_2, \ldots, p_{n+2}$ in this order; see \figurename~\ref{fig:5-gonaltiling} for an example.
For the special case of $t=3$, $T$ is a triangulation.
In the next section, we will show that the $k$-colored matchings on $2n$ vertices of the previous section correspond to $(k\!+\!2)$-gonal tilings of 
$n+2$ points in convex position, where $n=km$ for some integer $m>0$. 
This is a generalization of the fact that 
matchings (i.e., $k=1$) correspond to triangulations. 
To this end we first derive several properties of $t$-gonal tilings of convex sets.  

\begin{figure}[htb]
\centering\includegraphics[scale=1,page=2]{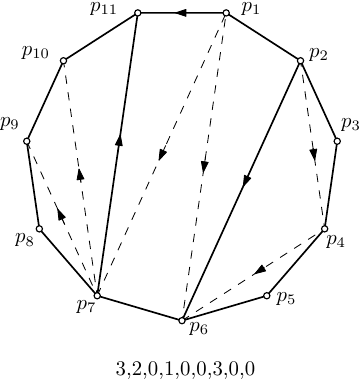}
\caption{5-gonal tiling corresponding to the 3-colored matching of \figurename~\ref{fig:matching_colored} and the outdegree sequence of its $k$-color valid triangulation.}
\label{fig:5-gonaltiling}
\end{figure}

The \emph{dual graph} of a $t$-gonal tiling $T$ has a vertex for each bounded face 
$T$ and two vertices are connected by an edge if the corresponding faces 
share a common edge in $T$ (every pair of bounded faces shares at most one edge). 
An \emph{ear} of $T$ is a $t$-gon which shares all but one edge with the unbounded face 
and can thus be cut off of $T$ (along this edge) so that the remaining part is a valid $t$-gonal tiling of $n+2-(t-2)=n+4-t$ points.

As the dual graph of any $t$-gonal tiling $T$ is a tree, as every tree with at least two vertices has at least two leaves (where the minimal case is obtained by a path), and as a leaf in the dual graph of $T$ corresponds to an ear in $T$, we have the following observation:
\begin{obs}
\label{obs:ear}
Every $t$-gonal tiling with at least $2t-2$ points has at least two ears. 
At least one of these ears is not incident to the edge $p_1 p_{n+2}$.
\end{obs}

\begin{lem}\label{lem:at_most_1_tiling}
Any triangulation $\cal T$ on $n+2$ points in convex position contains at most one $t$-gonal tiling as a subgraph.
\end{lem}

\begin{proof}
We prove the lemma by induction on $n$.
For $n+2=t$ the statement is obviously true, so let $n+2 \geq 2t-2$ and let $T_1$ and $T_2$ be two $t$-gonal tilings which are subgraphs of $\cal T$.  
By Observation~\ref{obs:ear} there exists an ear $E$ in $T_1$. 
Let $p_a p_b$, $a<b$, be the edge of $\cal T$ such that $E$ can be separated from the rest of $T_1$ by this edge. 
Moreover let $e$ be an edge that is incident to $E$ and to the unbounded face of $\cal T$. 
Then the (unique) $t$-gon in $T_2$ that is incident to $e$ must be $E$: Otherwise there is an edge connecting a point $p_x$ between $p_a$ and $p_b$  to a point $p_y$ outside the sequence from $p_a$ to $p_b$. 
Then $p_a p_b$, which is part of $T_1$, crosses $p_x p_y$, which is part of $T_2$. 
This is a contradiction to the planarity of $\cal T$ (recall that $T_1$ and $T_2$ are subgraphs of $\cal T$). 
Thus we can remove $E$ from both $T_1$ and $T_2$, and obtain two $t$-gonal tilings of a smaller set of points contained in the restriction of $\cal T$. 
By induction, these smaller $t$-gonal tilings are the same, and hence $T_1$ and $T_2$ are the same as well.
\end{proof}

Obviously, if a triangulation $\cal T$ on $n+2$ points contains a $t$-gonal tiling $T$ as a subgraph, then $n$ is a multiple of $t-2$. 
Further, as $T$ has at least two ears, $\cal T$ contains at least two edges that cut off a triangulated $t$-gon from ${\cal T}$. 
We call such a $t$-gon that can be split off from a triangulation $\cal T$ a \emph{$t$-ear} of $\cal T$ and refer to the edge along which the $t$-ear can be split off as an \emph{ear-edge} (of the $t$-ear).
Note that for $t>3$,  not every triangulation contains $t$-ears.

Let $\cal T$ be a triangulation that contains a $t$-ear with ear-edge $p_rp_s$ for some $r \geq 1$ and $s=r+t-1 \leq n+2$. 
Let $B$ be the outdegree sequence of the corresponding matching, obtained as described in Section~\ref{sec:uncolored}.
If $s < n+2$, then in $B$, the $t$-ear corresponds to a subsequence $W$ (obtained from $p_r,\ldots, p_{s-1}$) of $B$ of length $2t-3$ that 
  starts with a $1$ (for $p_rp_s$),  ends with two 0s (as the last point $p_{s-1}$ of the ear cannot have outgoing edges), and 
  has $t-1$ 0s and $t-2$ 1s in total. 
If $s = n+2$, then the point $p_{s-1} = p_{n+1}$ does not contribute to the outdegree sequence, cf.~Section~\ref{sec:uncolored}.
Thus the according subsequence $W$ has length $2t-4$ and is $W=(b_{2n-2t+5},\ldots, b_{2n})$, which must be a ballot sequence. 

\section{Relating $k$-colored matchings and $(k\!+\!2)$-gonal tilings}

We say that a triangulation on $n+2$ points in convex position is \emph{$k$-color valid} if by the bijection defined in Section~\ref{sec:uncolored} it corresponds to a $k$-colored matching as defined in Section~\ref{sec:coloredmatchings}. 
The outdegree sequence of such a triangulation is then also called $k$-color valid. 
A $(k\!+\!2)$-gonal tiling of $n+2$ points is called $k$-color valid if it can be completed to (i.e., is a subgraph of) a $k$-color valid triangulation. 
In the following, let $t=k+2$.

\begin{obs}\label{obs:t_ear_code}
	Let $\cal T$ be a $k$-color valid triangulation that contains a $t$-ear with ear-edge $p_rp_s$ for some $r \geq 1$ and $s=r+t-1 \leq n+2$.
	Let the first entry of the subsequence $W$ of $B$ that corresponds to this $t$-ear 
	be the $i$-th entry within its block, for $1\!\leq\!i\!\leq\!k$.
	If $s\!=\!n\!+\!2$ then $i=1$ and  $W=(|1,\ldots,1|0,\ldots,0|) = (|1^k|0^k|)$.
	Otherwise, recall from Section~\ref{sec:coloredmatchings} that within a block no 1 can be placed before a 0, and thus it holds that $W=(1,\ldots,1|0,\ldots,0,1,\ldots,1|0,\ldots,0) = (1^{k-i+1}|0^{k-i+1},1^{i-1}|0^{i})$. 
	In the former case, removing the t-ear is equivalent to removing $W$ from $B$. 
	In the latter case, all but the last $0$ of $W$ is removed from $B$.
\end{obs}

\begin{obs}\label{obs:t_ear_code_converse}
Using the same setting as in Observation~\ref{obs:t_ear_code} the converse also holds: if $B$ contains a subsequence $W= (1^{k-i+1}|0^{k-i+1},1^{i-1}|0^{i})$ or the end of B is $W = (|1^k|0^k|)$  then  $\cal T$ contains a $t$-ear.
\end{obs}
The following three lemmas can be derived using Observation~\ref{obs:t_ear_code}.
The proof of Lemma~\ref{lem:extend} also shows that the extension is uniquely determined.

\begin{lem}
\label{lem:extend}
Any $k$-color valid $t$-gonal tiling $T$ on $n+2$ points can be extended by an ear at any edge $e=p_rp_{r+1}$, $1 \leq r \leq n+1$, so that the resulting $t$-gonal tiling on $n+k$ points is $k$-color valid.
\end{lem}

\begin{proof}
Let $e=p_rp_{r+1}$ be the edge where we add the ear, and let $B$ be the outdegree sequence of the $k$-colored matching corresponding to $T$.
If $r \leq n$, then in $B$, $e$ corresponds to the 0, denoted here by $0'$, between the 1s that correspond to the outdegrees  $d_r$ and $d_{r+1}$ of $p_r$ and $p_{r+1}$, respectively, or the 0s of the  preceding (subsequent) boundary edge in case $d_r$ ($d_{r+1}$) is zero.
Suppose that $0'$ be the $i$-th entry within its block $R$, for some $1 \leq i \leq k$. 
Then $R = |0^{i-1},0',m|$, where $m$ is an arbitrary but valid subsequence. 
We extend $0'$ to a $t$-ear (by inserting $k$ 1s and $k$ 0s before $0'$ according to Observation~\ref{obs:t_ear_code}, by this extending $R$ to $|0^{i-1},1^{k-i+1}|0^{k-i+1},1^{i-1}|0^{i-1},0',m|$.
If $r = n+1$, then $e$ is not represented in $B$. 
In this case, we extend $B$ by adding a block of 1s followed by a block of 0s; see again Observation~\ref{obs:t_ear_code}.
In both cases, all $k$ new edges in the matching are local within the new blocks and monochromatic. Thus it follows by Theorem~\ref{thm:valid} that the extended outdegree sequence is also color valid.
Note that once $e$ is fixed, by Observation~\ref{obs:t_ear_code} the extension is uniquely determined. 
\end{proof}

\begin{lem}
\label{lem:cutoff}
Let $\cal T$ be a $k$-color valid triangulation that contains a $t$-ear with ear-edge $p_rp_s$ for some $r \geq 1$ and $s=r+t-1 \leq n+2$. 
Then the triangulation $\cal T'$ that results from removing the $t$-ear from $\cal T$ is again $k$-color valid.
\end{lem}

\begin{proof}
	Let $B$ be the outdegree sequence of the $k$-colored matching $M$ corresponding to $\cal T$ and let $W$ be the subsequence of $B$ corresponding to the $t$-ear.
	In $B$, the removal of the ear is equivalent to removing $W$ from $B$ (except for the last 0 for $s<n+2$). Let $W'$ be this sequence to be removed. 
	To show that the resulting triangulation $\cal T'$ is again $k$-color valid, we need to prove that the shortened outdegree sequence $B'$ corresponds to a  
	$k$-colored matching. 
	To this end, first note that in $M$, removing $W'$ from $B$ is equivalent to removing $2k$ consecutive vertices of the point set. 
	Hence the remaining vertices with the original $k$-coloring are properly colored. 
	Second, note that the number of 0s in $W'$ is $k$ and the number of 1s in $W'$ is $k$, implying that $B'$ corresponds to some matching $M'$.
	It remains to show that $M'$ is $k$-colored, that is, that there is no bichromatic edge in $M'$. 
	By Observation~\ref{obs:t_ear_code}, we have $W'=(1^{k-i+1}|0^{k-i+1},1^{i-1}|0^{i-1})$ for some $1\!\leq\!i\!\leq\!k$. 
	In the matching $M$, this corresponds to $k$ edges that form a matching of the vertices to be removed. 
	Hence all edges in $M'$ also exist in $M$, implying that none of them is bichromatic.
\end{proof}

\begin{lem}
\label{lem:earexistence}
Let $\cal T$ be a $k$-color valid triangulation. Then $\cal T$ contains a $t$-ear with ear-edge $p_rp_s$ for some $r \geq 1$ and $s=r+t-1 \leq n+2$. 
\end{lem}

\begin{proof}
	Let $B$ be the outdegree sequence of the $k$-colored matching corresponding to $\cal T$.
	Further, let $W_i$ be the subsequence of $B$ that starts at $b_i$ and has length $2k+1$, for $1 \leq i \leq 2n-2k$, and let $w_i = \sum_{j=i}^{i+2k} b_j$ be the \emph{weight} of $W_i$.
	As $\cal T$ is $k$-color valid, we have $w_1 > k$ (there have to be at least $k+1$ outgoing edges for the first $2k+1$ vertices) and $w_{2n-2k} \leq k$ (there are at most $k$ outgoing edges for the last $2k+1$ vertices). Further, we also have $w_{i+1}-w_i \in \{0,\pm1\}$.
	We will show that either at least one of the $W_i$s or the last two blocks of $B$ represents a $k$-ear of $\cal T$. 
	To this end, we proceed through the $W_i$s from $i=1$ to $2n-2k$ as long as $w_i \geq k$. 
	Whenever $w_i > k$, we continue to the next subsequence (as a necessary condition for $W_i$ to be a $k$-ear is $w_i=k$).
	For $w_i=k$ and $w_{i-1} > k$, $W_{i-1}$ starts with $b_{i-1}=1$ and $W_i$ ends with $b_{i+2k}=0$. 
	We distinguish the following cases:\\ 
	{\bf Case 1.~} $W_i$ starts with $b_i=1$. 
		Let $1 \leq a \leq k$ be such that the block containing $b_i$ ends right before $b_{i+a}$.
		Then we have $W_i=1^a|0^a1^{k-a}|0^{k-a+1}$, where the 1s in the first block are forced by $b_i=1$, the 0s in the last block are forced by $b_{i+2k}=0$, and the form of the middle block stems from $w_i=k$.
		Hence, $W_i$ is a $k$-ear by Observation~\ref{obs:t_ear_code_converse}. \\
	{\bf Case 2.~} $W_i$ starts with $b_i=0$. As $W_{i-1}$ starts with $b_{i-1}=1$, 
		there is a block boundary directly before $b_i$, and by $w_i=k$ we have $W_i=|0^a1^{k-a}|0^{k-a}1^a|0$ for some $1 \leq a \leq k$.
		Hence, $W_j$ is no ear and $w_j \geq k$ for $i \leq j \leq \min\{i+a,2n-2k\}$.\\
		{\bf Case 2.1.~} If $i+a \leq 2n-2k$ and $w_{i+a}>k$ then $i+a < 2n-2k$ and we continue the whole process by considering $w_{i+a+1}$.\\
		{\bf Case 2.2.~} If $i+a \leq 2n-2k$ and $w_{i+a}=k$ then all entries in  $W_{i+a}\setminus W_i$ are 0s and hence $W_{i+a}=1^{k-a}|0^{k-a}1^a|0^{a+1}$ is a $k$-ear by Observation~\ref{obs:t_ear_code_converse}.\\
		{\bf Case 2.3.~} If $i+a > 2n-2k$, then all 1s in $W_i$ must also be in $W_{2n-2k}$.
			Thus $w_{2n-2k}=k$ and due to the $k$-color validity we have $W_{2n-2k}= 0|1^k|0^k|$.
			Hence the last two blocks of $B$ form a $k$-ear by Observation~\ref{obs:t_ear_code_converse}.
\end{proof} 

Combining Lemmas~\ref{lem:at_most_1_tiling}~--~\ref{lem:earexistence} and Observations~\ref{obs:ear}~--~\ref{obs:t_ear_code_converse}, we obtain our main result. 

\begin{thm}\label{thm:bij-k}
	For integers $k \geq 2$ and $c \geq 1$ let $n=ck$ and $t=k+2$.
	There exists a bijection between geometric plane perfect $k$-colored matchings on $2n$ points in convex position and $t$-gonal tilings on $n+2$ points in convex position.
	Further, for an element of one structure, the corresponding element of the other structure can be computed in linear time.
\end{thm}

\begin{proof}
We first show (by induction on $n$) that every $t$-gonal tiling $T$ can be completed to at least one $k$-color valid triangulation.
For $n+2=t$ the statement is trivially true as we have only one inner face and can thus triangulate as required.
So let $n+2 \geq 2t-2$. 
By Observation~\ref{obs:ear} there exists an ear $E$ of $T$. 
If we cut this ear off, then by induction there exists a completion to a $k$-color valid triangulation, which by Lemma~\ref{lem:extend} can be extended to a $k$-color valid triangulation $\cal T$ of $T$.

Next, assume that there exists a $t$-gonal tiling which can be refined by at least two different $k$-color valid triangulations.
Let $T$ be a minimal such $t$-gonal tiling and let ${\cal T}_1$ and ${\cal T}_2$ be two different $k$-color valid triangulations for $T$.
By Lemma~\ref{lem:earexistence}, ${\cal T}_1$ has a $t$-ear with ear-edge $e=p_rp_s$ for some $r \geq 1$ and $s=r+t-1 \leq n+2$. 
Thus, $e$ must be an edge of $T$, implying that ${\cal T}_2$ also has a $t$-ear at $e$. 
By Lemma~\ref{lem:cutoff}, removing the $t$-ear from ${\cal T}_1$ results in a $k$-color valid triangulation ${\cal T}'$.
Further, as $T$ is minimal, removing the $t$-ear from ${\cal T}_2$ results in the same triangulation ${\cal T}'$.
But by the proof of Lemma~\ref{lem:extend}, there is exactly one possibility of extending ${\cal T}'$ at $e$ with a $t$-ear, a contradiction.
Hence every $t$-gonal tiling $T$ can be completed to exactly one $k$-color valid triangulation.

So far we have shown that a given $t$-gonal tiling can be completed to exactly one $k$-color valid triangulation. 
For proving that there exists a bijection between $k$-colored matchings and $t$-gonal tilings, it remains to show that any $k$-color valid triangulation contains exactly one $t$-gonal tiling.

We show (by induction on $n$) that every $k$-color valid triangulation ${\cal T}$ contains at least one $t$-gonal tiling.
For $n+2=t$ the statement is trivially true, so let $n+2 \geq 2t-2$. 
By Lemma~\ref{lem:earexistence}, ${\cal T}$ has a $t$-ear with ear-edge $e=p_rp_s$ for some $r \geq 1$ and $s=r+t-1 \leq n+2$.
Further, by Lemma~\ref{lem:cutoff}, removing the $t$-ear from ${\cal T}$ results in a triangulation ${\cal T}'$, which, by induction, contains at least one $t$-gonal tiling $T'$. By Lemma~\ref{lem:extend}, we can extend $T'$ with an ear at $e$, thus obtaining a $t$-gonal tiling for ${\cal T}$. 

As by Lemma~\ref{lem:at_most_1_tiling}, every $k$-color valid triangulation ${\cal T}$ contains at most one $t$-gonal tiling $T$, this completes the proof of the existence of a bijection. 

To show that the transformation from a $k$-colored matching to a $t$-gonal tiling and vice versa can be done in linear time,
it remains to show that the $t$-gonal tiling of a $k$-color valid triangulation can be found in linear time and vice versa.

Consider first a $k$-color valid triangulation ${\cal T}$, let $B$ be the outdegree sequence of the $k$-colored matching corresponding to ${\cal T}$, and let $B$ be stored in a linked list. Let $T$ be the $t$-gonal tiling for ${\cal T}$ that we want to construct.
By the proof of Lemma~\ref{lem:earexistence}, we find a $t$-ear of ${\cal T}$ whose subsequence $W$ in $B$ starts at $b_j$ and 
which is the first $t$-ear of ${\cal T}$ in time $O(j+2k)$. 
We can remove the $t$-ear from ${\cal T}$ and $W$ (except possibly its last 0) from $B$ in constant time, by this also obtaining one diagonal of $T$.
Further, the first ear in the shortened sequence can start at earliest at $b_{j-2k}$, which implies that we do not need to restart our scan at the beginning.
Hence, we can iteratively find all diagonals of $T$ in $O(n)$ time.

For the other direction, consider a $t$-gonal tiling. We recursively cut off all ears in total linear time. 
Then, using Lemma~\ref{lem:extend}, we re-add them in reverse order,  together with their triangulations that are uniquely defined by Observation~\ref{obs:t_ear_code}.
\end{proof}

\section{Future Work}

It is natural to search for a characterization of the generators of Temperley-Lieb algebras 
in terms of triangulations (and for the generators for the $k$-colored Fuss-Catalan algebras in terms of $(k\!+\!2)$-gonal tilings). 
We plan to use our explicit bijections to study the effect of edge flips in triangulations (respectively in tilings) on the corresponding matchings 
and to find out how the actions of generators of the Temperley-Lieb algebra (the $k$-colored Fuss-Catalan algebra) can be interpreted in terms of flips in triangulations respectively in tilings. 
Preliminary results have already been obtained.
%

\paragraph{Acknowledgements.} 
Research for this work is supported by the Austrian Science Fund (FWF) grant W1230. 
We thank Paul Martin for bringing this problem to our attention.

\bibliographystyle{abbrv}
\bibliography{bibliography}

\end{document}